\documentclass[11pt,a4paper]{article}

\usepackage[latin1]{inputenc}
\usepackage{latexsym,amssymb}
\usepackage{amsmath,amssymb,amsfonts,amscd,theorem}
\usepackage{enumerate}
\usepackage[francais,english]{babel}

\renewcommand{\hom}{\operatorname{Hom}}

\newcommand{\CC}{\mathbb{C}}

\newcommand{\NN}{\mathbb{N}}

\newcommand{\ZZ}{\mathbb{Z}}
\newcommand{\card}{\operatorname{card}}
\newcommand{\vect}{\operatorname{Vect}}

\def\co{\colon\thinspace}
\newenvironment{proof}[1][Preuve]%
{
\begin{trivlist} \item[]  {\em #1.} }%
{\hspace*{\fill} $\Box$
\end{trivlist}}

{\bigl(
\begin{smallmatrix} }%
{
\end{smallmatrix} \bigr)}

\newtheorem{thm}{Th\'eor\`eme}
\newtheorem{prop}[thm]{Proposition}
\newtheorem{lemma}[thm]{Lemme}

\newtheorem{cor}[thm]{Corollaire}
\newtheorem{rk}[thm]{Remarque}

\title{Module d'Alexander et repr\'esentations m\'etab\'eliennes}

\author{Hajer Jebali\\
Faculté des Sciences de Monastir\\
Boulevard de l'environnement\\
5019 Monastir\\
Tunisie\\
courriel : hajer.jebali@fsm.rnu.tn}
\date{}

\begin{document}

\selectlanguage{francais}

\maketitle

\begin{abstract}

On sait, depuis des travaux de Burde et de Rham, que l'\'etude
des repr\'esentations du groupe d'un n\oe ud dans le groupe des
matrices triangulaires sup\'erieures inversibles d'ordre $2$
permet de d\'etecter les racines du polyn\^ome d'Alexander du n\oe
ud. Dans ce travail, nous nous proposons de g\'en\'eraliser ce
r\'esultat et ce en consid\'erant les repr\'esentations
 du groupe du n\oe ud dans le groupe des matrices triangulaires sup\'erieures inversibles d'ordre $n,\ n\geq 2$.
  Cette approche nous permettra de retrouver la d\'ecomposition du module d'Alexander \`a c\oe fficients complexes du n\oe ud.
\medskip
\noindent 

\end{abstract}

\selectlanguage{english}
\begin{abstract}

It is known, since works of Burde and de Rham, that one can detect the roots of the Alexander polynomial of a knot by  studying the representations of the knot group into the group of the invertible upper triangular $2\times 2$ matrices. In this work, we propose to generalize this result by considering the representations of the knot group into the group of the invertible upper triangular $n\times n$ matrices, $n\geq 2$. This approach will enable us to find the decomposition of the Alexander module with complex c\oe fficients of the knot.
\end{abstract}

\selectlanguage{francais}

\section{Introduction}
\label{sec:intro}

Soient $K$ un n\oe ud de $S^3$, $V(K)$ un voisinage tubulaire de $K$, $X=\overline{S^3\setminus V(K)}$ son compl\'ementaire et $\pi=\pi_1(\overline{S^3\setminus V(K)})$ le groupe fondamental de $X$. Soit $\mu$ un m\'eridien du n\oe ud. Notons $X^\infty$ le rev\^etement cyclique infini de $X$ correspondant au groupe des commutateurs $\pi'=[\pi,\pi]$ et $\pi/\pi'\simeq T=\left\langle t|-\right\rangle$ le groupe quotient monog\`ene infini engendr\'e par l'image $t$ du m\'eridien $\mu$. Les groupes d'homologie $H_*(X^\infty,\CC)$ sont munis d'une structure de $\Lambda=\CC[t,t^{-1}]$-modules. Ces modules sont des modules de torsion et de type fini. Le module $H_1(X^\infty,\CC)$ est appel\'e module d'Alexander \`a c\oe fficients complexes du n\oe ud. Son id\'eal d'ordre est principal. Tout g\'en\'erateur de cet id\'eal est appel\'e polyn\^ome d'Alexander de $K$ et est not\'e $\Delta_K(t)$ (voir \cite{Gor78}).

Le module d'Alexander \`a c\oe fficients complexes se d\'ecompose sous la forme

\[
H_{1}(X^{\infty},\CC)=\displaystyle\bigoplus_{\Delta_{K}(\alpha)=0}\tau_{\alpha},
\text{ avec }
\tau_{\alpha}=\displaystyle\bigoplus_{i=1}^{k_{\alpha}}\tau_{\alpha}^{i},
\text{ o\`u }
\tau_{\alpha}^{i}=\displaystyle\frac{\Lambda}{(t-\alpha)^{q_{i}}},
\]
$q_{i}\in\NN^{*}$ et 
$k_{\alpha}=\dim_\CC 
(H^{1}(\pi,\CC_{\alpha})\otimes_\Lambda \CC_\alpha)$.
Ici, pour $\alpha\in\CC^{*}$, on notera $\CC_{\alpha}$ le
$\Lambda$-module $\CC$ dont l'action est donn\'ee par
$$q(t)\cdot z=q(\alpha)z,\quad q(t)\in\Lambda.$$
Le $\Lambda$-module $\CC_{1}$ est simplement not\'e $\CC$.

Soit $G_n$ le groupe des matrices triangulaires sup\'erieures inversibles d'ordre $n$, $n\geq 2$. On regarde le produit semi-direct $H_1(X^\infty,\CC)\rtimes T$, qui s'injecte dans $\displaystyle\prod_{\Delta_K(\alpha)=0}\prod_{i=1}^{k_\alpha}\frac{\Lambda}{(t-\alpha)^{q_i}}\rtimes T$. Or, un facteur de la forme $\displaystyle\frac{\Lambda}{(t-\alpha)^q}\rtimes T$ s'injecte dans $G_{q+1}$ (voir~\cite{BF08}), d'o\`u l'int\'er\^et de consid\'erer les repr\'esentations, c'est-\`a-dire homomorphismes de groupes, du groupe du n\oe ud \`a valeurs dans $G_n$, $n\geq 2$. Plus pr\'ecis\'ement, nous nous int\'eressons \`a l'\'etude des repr\'esentations qui sont de la forme
\[\rho_n(\gamma)=\begin{pmatrix}
\alpha^{|\gamma|}&x_{12}(\gamma)&x_{13}(\gamma)&\ldots&x_{1n}(\gamma)\\
0&1&x_{23}(\gamma)&\ldots&x_{2n}(\gamma)\\
\vdots &\ddots&\ddots&&\vdots\\
\vdots&&\ddots&1&x_{n-1,n}(\gamma)\\
0&\ldots&\ldots&0&1
\end{pmatrix}\,,\]
avec $\mid\gamma\mid=p(\gamma)$, o\`u $p\co\pi\to\pi/\pi'\simeq\ZZ$ d\'esigne la projection canonique.

Ce choix a \'et\'e motiv\'e, d'une part, par les travaux de~\cite{Dwy75},~\cite{FS87} et~\cite{Mor04} qui ont \'etabli un lien entre le produit de Massey et l'\'etude des homomorphismes de groupes, des suites centrales descendantes des groupes libres, du calcul diff\'erentiel libre ou encore des invariants de Milnor. Le produit de Massey est d\'efini \`a l'aide des matrices triangulaires sup\'erieures avec $1$ sur la diagonale. D'autre part, ce choix a \'et\'e motiv\'e par les travaux de Burde et de Rham qui se sont int\'eress\'es au cas $n=2$ \cite{Bur67} et \cite{dR67}. Plus pr\'ecis\'ement, ils ont s\'epar\'ement montr\'e qu'il existe des repr\'esentations  non ab\'eliennes du groupe du n\oe ud dans $G_2$ si et seulement si $\alpha$ est racine du polyn\^ome d'Alexander. Ils ont \'egalement montr\'e que $k_\alpha=\dim H^1(\pi,\CC_\alpha)$.

Soit, pour $n\geq 2$, l'application $\rho_n^V$ qui \`a $\gamma\in\pi$ associe la matrice
\[ \rho_{n}^{V} (\gamma) =\left(\begin{array}{c|cc}
\alpha^{|\gamma|}&V(\gamma)\\
\hline 0& \phi_{n-1}(\gamma)&
\end{array}\right)\]
o\`u $V = (v_{1},\ldots,v_{n-1})$ est un $(n-1)$-uplet de $1$-cocha\^ines de $\pi$ dans le $\pi$-module $\CC_\alpha$ (voir paragraphe~\ref{sec:eqobs}) et $\phi_m\co\pi\to GL(m,\CC)$ est l'homomorphisme ab\'elien d\'efini par 
\[\phi_m(\mu)=\begin{pmatrix}
1&1&0&\cdots&0\\
0&1&1&\ddots&\vdots\\
\vdots& \ddots&\ddots&\ddots &0\\
\vdots& &\ddots &1&1\\
 0&\cdots&\cdots
 &0&1 \end{pmatrix}\,.\]
 Posons
 
 \[C_n=\{ U\in Z^{1}(\pi,\CC_{\alpha})\ |\ \exists\ (v_{2},\ldots,v_{n-1})\in(C^{1}(\pi,\CC_{\alpha}))^{n-2},\]
 
 \[\ \rho_{n}^{(U,v_{2},\ldots,v_{n-1})}\in\hom(\pi,G_n)\}\,.\]

Nous montrerons que $C_n$ est un sous-espace vectoriel de l'ensemble des $1$-cocycles $Z^1(\pi,\CC_\alpha)$ qui contient l'ensemble des $1$-cobords $B^1(\pi,\CC_\alpha)$. Plus pr\'ecis\'ement, nous montrerons le r\'esultat suivant:
\begin{thm}\label{thm:fond}
Soit $n\geq 2$, alors l'ensemble des $1$-cobords $B^1(\pi,\CC_\alpha)$ est strictement contenu dans $C_n$ si et seulement s'il existe $1\leq i\leq k_\alpha$ tel que $q_i\geq n-1$.

De plus, $\dim C_n=1+\card\{1\leq i\leq k_\alpha\ |\ q_i\geq n-1\}\,.$
\end{thm}

Si nous supposons que les puissances $q_i$ sont ordonn\'ees de sorte que
\[1\leq q_1\leq\ldots\leq q_{k_\alpha}\]
alors nous obtenons le r\'esultat suivant:

\begin{cor}
\begin{enumerate}
\item Dans la filtration de l'espace des $1$-cocycles $Z^1(\pi,\CC_\alpha)$, on a:
$$B^{1}(\pi,\CC_{\alpha})= C_{q_{k_{\alpha}}+2}\subsetneq C_{q_{k_{\alpha}}+1}\subseteq
C_{q_{k_{\alpha}}}\subseteq\cdots\subseteq
C_{2}=Z^{1}(\pi,\CC_{\alpha})\,.$$

\item
La codimension du sous-espace $C_{p}$ dans $C_{p-1}$ est
\'egale au nombre des $q_{i}$ \'egaux \`a $p-2$, c'est-\`a-dire
$$\dim C_{p-1}-\dim C_{p}= \card\{1\leq i\leq k_\alpha\ |\ q_{i}=p-2\}\,,\ \forall\ p\geq 3\,.$$
\end{enumerate}
\end{cor}

La suite descendante des sous-espaces vectoriels $(C_p)_{p\geq 2}$ ainsi d\'ecrite nous permet donc de retrouver la d\'ecomposition du module d'Alexander \`a c\oe fficients complexes du n\oe ud.

Remarquons qu'une correction de $U\in Z^1(\pi,\CC_\alpha)$ par un cobord se traduit par une conjugaison de $\rho_n^{(U,v_2,\ldots,v_{n-1})}$ et posons
\[\overline{C}_n=\{ \{U\}\in H^{1}(\pi,\CC_{\alpha})\ |\ \exists\ (v_{2},\ldots,v_{n-1})\in(C^{1}(\pi,\CC_{\alpha}))^{n-2},\]
 
 \[\ \rho_{n}^{(U,v_{2},\ldots,v_{n-1})}\in\hom(\pi,G_n)\}\,.\]
 Alors, on a:
 \begin{cor}
 \begin{enumerate}
 \item L'espace vectoriel $\overline{C}_n$ est non r\'eduit \`a z\'ero si et seulement s'il existe $1\leq i\leq k_\alpha$ tel que $q_i\geq n-1$.
 \item Dans la filtration de l'espace de la premi\`ere classe de cohomologie $H^1(\pi,\CC_\alpha)$, on a:
 \[\{0\}=\overline{C}_{q_{k_{\alpha}}+2}\subsetneq \overline{C}_{q_{k_{\alpha}}+1}\subseteq
\overline{C}_{q_{k_{\alpha}}}\subseteq\cdots\subseteq
\overline{C}_{2}=H^{1}(\pi,\CC_{\alpha})\,.\]
\item
La codimension du sous-espace $\overline{C}_{p}$ dans $\overline{C}_{p-1}$ est
\'egale au nombre des $q_{i}$ \'egaux \`a $p-2$, c'est-\`a-dire
$$\dim \overline{C}_{p-1}-\dim \overline{C}_{p}= \card\{1\leq i\leq k_\alpha\ |\ q_{i}=p-2\}\,,\ \forall\ p\geq 3\,.$$
\end{enumerate}
\end{cor}

Les repr\'esentations $\rho_n^V$ sont m\'etab\'eliennes, c'est-\`a-dire leurs restrictions au deuxi\`eme sous-groupe des commutateurs $\pi''=[\pi',\pi']$ sont triviales. Un travail r\'ecent de~\cite{BF08} a port\'e sur la classification des repr\'esentations m\'etab\'eliennes irr\'eductibles du groupe d'un n\oe ud dans $SL(n,\CC)$ et $GL(n,\CC)$.

\begin{rk}
La matrice $\rho_n^V(\gamma)$, pour $\gamma$ dans $\pi$, est un syst\`eme de d\'efinition du produit de Massey $< U,\underbrace{h_1,\ldots,h_1}_{(n-2)}>$ (voir~\cite{Kra66} et~\cite{FS87}). L'ensemble $\overline{C}_n$ peut \^etre d\'efini en fonction du produit de Massey:
\[\overline{C}_n=\{\{U\}\in H^1(\pi,\CC_\alpha)\mid < U,\underbrace{h_1,\ldots,h_1}_{(n-2)}>=0\}\,.\]
\end{rk}

\section{Equations d'obstruction}
\label{sec:eqobs}

Soient $\Gamma$  un groupe de pr\'esentation finie et $M$ un $\Gamma$-module \`a gauche. L'espace des $n$-cocha\^ines $C^n(\Gamma,M)$ du groupe $\Gamma$ dans le module $M$, pour $n\geq 0$, est par d\'efiniton l'espace des fonctions $f$ de $\Gamma^n$ dans $M$. L'op\'erateur $\delta\co C^n(\Gamma,M)\to C^{n+1}(\Gamma,M)$ est donn\'e en petites dimensions par~\cite{Bro82}:
\[\delta a(\gamma)=\gamma\cdot a-a,\ \text{pour tout}\ a\in M\,,\]
\[\delta f(\gamma_1,\gamma_2)=\gamma_1\cdot f(\gamma_2)-f(\gamma_1\gamma_2)+f(\gamma_1),\ \text{pour tout}\ f\in C^1(\Gamma,M)\,.\]

Si $U$ d\'esigne un cocycle dans $Z^i(\Gamma,M)$, $i\geq 1$, on notera par $\{U\}$ sa classe de cohomologie dans $H^i(\Gamma,M)$.

Soient $\alpha,\ \beta\in\CC^*$, les $\Lambda$-modules $\CC_\alpha$ et $\CC_\beta$ sont munis de la structure de $\pi$-modules via la projection $\pi\twoheadrightarrow T$ et l'action est donn\'ee par
\[\gamma\cdot z=\alpha^{|\gamma|}z\,,\quad\forall\ \gamma\in\pi\ \text{et}\ \forall\ z\in\CC\,.\]
Si $f\in C^1(\pi,\CC_\alpha)$ et $g\in C^1(\pi,\CC_\beta)$ d\'esignent deux cocha\^ines \`a valeurs dans les $\pi$-modules $\CC_\alpha$ et $\CC_\beta$, alors on notera $f\cup g\in C^2(\pi,\CC_{\alpha\beta})$ leur produit-cup donn\'e par:
\[f\cup g(\gamma_1,\gamma_2)=f(\gamma_1)\beta^{|\gamma_1|}g(\gamma_2)\,,\ \gamma_1,\ \gamma_2\in\pi\,.\]

Notons $N_m$ le groupe des matrices triangulaires sup\'erieures d'ordre $m$ avec $1$ sur la diagonale. Le groupe $N_m$ \'etant un groupe de Lie nilpotent, toutes les repr\'esentations de $\pi$ \`a valeurs dans $N_m$ sont ab\'eliennes. Ce r\'esultat peut \^etre retrouv\'e en utilisant un r\'esultat classique de Milnor~\cite{Mil54}. En effet, Milnor montre que la suite centrale descendante du groupe d'un n\oe ud est stationnaire. Plus pr\'ecis\'ement, il montre que $\mathcal{C}^k\pi=\mathcal{C}^1\pi=\pi'$, pour tout $k\geq 1$, o\`u $(\mathcal{C}^k\pi)_{k\geq 0}$ d\'esigne la suite centrale descendante de $\pi$.

Puisque $\pi/\pi'\simeq T=\left\langle t|-\right\rangle$, toute repr\'esentation ab\'elienne se factorise par $T$ et est compl\`etement d\'etermin\'ee par la donn\'ee de l'image de $\mu$.
Consid\'erons $\phi_{m}\co\pi\to N_{m}$ l'homomorphisme
ab\'elien d\'efini par
 
 \[\phi_m(\gamma)=\begin{pmatrix}
h_0(\gamma)&h_1(\gamma)&h_2(\gamma)&\cdots&h_{m-1}(\gamma)\\
0&h_0(\gamma)&h_1(\gamma)&\ddots&\vdots\\
\vdots& \ddots&\ddots&\ddots &h_2(\gamma)\\
\vdots& &\ddots &h_0(\gamma)&h_1(\gamma)\\
 0&\cdots&\cdots
 &0&h_0(\gamma) \end{pmatrix}\]
 et
\[\phi_m(\mu)=\begin{pmatrix}
1&1&0&\cdots&0\\
0&1&1&\ddots&\vdots\\
\vdots& \ddots&\ddots&\ddots &0\\
\vdots& &\ddots &1&1\\
 0&\cdots&\cdots
 &0&1 \end{pmatrix}=I_{m}+\mathcal{N}_{m}\]
 o\`u $I_{m}$ d\'esigne la matrice identit\'e.
 Alors, pour tout $k\in\ZZ$: 

 \begin{equation}\label{repab}
 \phi_{m}(\mu^{k})=(I_{m}+\mathcal{N}_{m})^{k}
 =\sum_{p=0}^{k}\binom{k}{p}\mathcal{N}_{m}^{p}
 = \sum_{p=0}^{k}h_{p}(\mu^{k})\mathcal{N}_{m}^{p}
 \end{equation} 
o\`u  $\displaystyle\binom{k}{p}$, $p\in\NN$, d\'esigne le
c\oe fficient binomial d\'efini par 
\[\binom{k}{0}:=1 \text{ et } \binom{k}{p}
:= \frac{k(k-1)\cdots (k-p+1)}{p!}\in\ZZ\,.\]
Les applications $h_p\co\pi\to\CC$, $p\geq 1$, sont solutions des \'equations d'obstruction:
\[\delta h_p+\displaystyle\sum_{i=1}^{p-1}h_i\cup h_{p-i}=0\]
et $h_0\in\hom(\pi,\CC^*)$.

Soit $n\geq 2$ et soit $\rho_n^V\co\pi\to G_n$ d\'efinie par

\[ \rho_{n}^{V} (\gamma) =\left(\begin{array}{c|cc}
\alpha^{|\gamma|}&V(\gamma)\\
\hline 0& \phi_{n-1}(\gamma)&
\end{array}\right)\]
Alors, $\phi_{n-1}$ \'etant un homomorphisme ab\'elien, une telle repr\'esentation $\rho_n^V$, lorsqu'elle existe, est m\'etab\'elienne, c'est-\`a-dire sa restriction au deuxi\`eme sous-groupe des commutateurs $\pi''=[\pi',\pi']$ est triviale. En effet, pour le prouver, il suffit de v\'erifier que $\rho_n^V|_{\pi'}$ est ab\'elienne. Or, pour tout $\gamma\in\pi'$, $\rho_n^V(\gamma)$ est une matrice de la forme
\[\begin{pmatrix}
1&*\\
0&I_{n-1}
\end{pmatrix}\]
car la repr\'esentation obtenue \`a partir de $\rho_n^V$ en supprimant la premi\`ere ligne et la premi\`ere colonne est ab\'elienne et deux matrices de cette forme commutent. Donc $\rho_n^V$ est m\'etab\'elienne.

Maintenant, une condition n\'ecessaire et suffisante pour que $\rho_n^V$ soit un homomorphisme se traduit par $\rho_n^V(\gamma_1\gamma_2)=\rho_n^V(\gamma_1)\rho_n^V(\gamma_2)$, pour tous $\gamma_1,\ \gamma_2\in\pi$, et est \'equivalente \`a
\begin{equation}\label{equ:hom}
V(\gamma_{1}\gamma_{2})=\alpha^{|\gamma_{1}|}V(\gamma_{2})+V(\gamma_{1})\phi_{n-1}(\gamma_{2})\,,\quad\forall\ \gamma_{1},\gamma_{2}\in\pi.
\end{equation}

On en d\'eduit que $C_n$ est un sous-espace vectoriel de $Z^1(\pi,\CC_\alpha)$. De plus, remarquons que s'il existe $b_0\in C^0(\pi,\CC_\alpha)$ tel que $v_1=\delta b_0$, alors $v_i=-b_0\cup h_{i-1}$, pour $2\leq i\leq n-1$, donne un $(n-1)$-uplet $V=(v_1,\ldots,v_{n-1})$ solution de~(\ref{equ:hom}). Donc $C_n$ est un sous-espace vectoriel de $Z^1(\pi,\CC_\alpha)$ qui contient l'ensemble des $1$-cobords $B^1(\pi,\CC_\alpha)$.

La condition d'homomorphie (\ref{equ:hom}) de $\rho^V_{n}$ est \'equivalente au syst\`eme suivant
\[ (S)\left\{\begin{array}{lc}
v_{1}\in Z^{1}(\pi,\CC_{\alpha})\\
-\delta v_{i}=\displaystyle\sum_{p=1}^{i-1}v_{p}\cup h_{i-p}& ,\forall\ 2\leq i\leq n-1
\end{array}\right.\]
(voir aussi \cite{FS87}).

En conclusion, nous cherchons un vecteur $V=(v_1,v_2,\ldots,v_{n-1})$ de l'espace vectoriel $(C^1(\pi,\CC_\alpha))^{n-1}$ v\'erifiant
\begin{equation}\label{eqobs}
\delta v_{i}+\displaystyle\sum_{p=1}^{i-1}v_{p}\cup h_{i-p}=0,\ \text{pour}\ i=1,\ldots,n-1
\end{equation}

Comme $H^{2}(\pi,\CC_{\alpha})=0$ si et seulement si $\Delta_{K}(\alpha)\neq 0$ \cite[prop.~2.1]{BA00}, un tel vecteur existe lorsque $\alpha$ n'est pas racine du polyn\^ome d'Alexander. Dans le cas o\`u $\alpha$ est racine du polyn\^ome d'Alexander du n\oe ud, nous utiliserons le r\'esultat suivant pour la r\'esolution des \'equations d'obstruction.

\begin{prop}\label{propphi}
Soit $n\geq 2$ et soit $\alpha\in\CC^{*}\backslash\{1\}$ une racine du polyn\^ome d'Alexander du n\oe{}ud. 

Alors, il existe un uplet $V=(v_{1},\ldots,v_{n-1})$ vérifiant (\ref{eqobs}) si et seulement s'il existe une famille
d'homomorphismes de groupes 
\[ \varphi_{i}\co\pi'/\pi''\to (\CC,+),\quad 1\leq i\leq n-1\]  
v\'erifiant 
\begin{equation}\label{applin}
\varphi_{i}(\mu^{k}y\mu^{-k})=\alpha^{k}\displaystyle\sum_{p=1}^{i}\binom{-k}{i-p}\varphi_{p}(y), \forall\ y\in\pi'/\pi''\ \text{et}\ \forall\ k\in\ZZ\,.
\end{equation}
\end{prop}

\begin{proof}
Commen\c cons par remarquer que l'\'equation~(\ref{applin}) peut se mettre sous la forme matricielle
\begin{equation}\label{Phi}
\boldsymbol{\Phi}(\mu^ky\mu^{-k})=\alpha^k\boldsymbol{\Phi}(y)J_{n-1}^{-k}\ ,\quad\forall\ y\in\pi'/\pi''\ \text{et}\ \forall\ k\in\ZZ
\end{equation}
o\`u $\boldsymbol{\Phi}$ d\'esigne le vecteur ligne $(\varphi_1,\ldots,\varphi_{n-1})\co\pi'/\pi''\to\CC^{n-1}$ et $J_{n-1}:=\phi_{n-1}(\mu)$ d\'esigne la matrice introduite au d\'ebut de ce paragraphe.

Supposons que $V$ existe. Rappelons que s'il existe $b_0\in C^0(\pi,\CC_\alpha)$ tel que $v_1=\delta b_0$ alors $v_i=-b_0\cup h_{i-1}$, $2\leq i\leq n-1$, donne une famille de $1$-cocha\^ines v\'erifiant~(\ref{eqobs}). Plus g\'en\'eralement, si $v_1=\ldots=v_{k-1}=0$ et si $v_k=\delta b_0$, alors $v_{k+i}=-b_0\cup h_i$, pour $1\leq i\leq n-1$, donne une telle famille. Donc nous pouvons supposer que $v_i(\mu)=0$, pour tout $1\leq i\leq n-1$. Chacune des cocha\^ines $v_i$ (resp. $h_i$) est m\'etab\'elienne donc elle passe au quotient par $\pi''$ et d\'efinit ainsi une $1$-cocha\^ine de $\pi'/\pi''$ \`a valeurs dans $\CC_\alpha$ (resp. $\CC$). Consid\'erons, pour $1\leq i\leq n-1$, les applications $\varphi_i=v_i|_{\pi'/\pi''}$. Comme $V$ v\'erifie~(\ref{equ:hom}), pour tout $\gamma\in\pi'$ et tout $k\in\ZZ$, on a:
\begin{align*}
V(\mu^k\gamma\mu^{-k})&=\alpha^kV(\mu^{-k})+V(\mu^k\gamma)\phi_{n-1}(\mu^{-k})\\
&=(\alpha^kV(\gamma)+V(\mu^k)\phi_{n-1}(\gamma))\phi_{n-1}(\mu^{-k})\\
&=\alpha^kV(\gamma)\phi_{n-1}(\mu^{-k})\,.
\end{align*}

D'o\`u
\[\boldsymbol{\Phi}(\mu^k[\gamma]\mu^{-k})=\alpha^k\boldsymbol{\Phi}([\gamma])\phi_{n-1}(\mu^{-k})\,,\]
o\`u $[\gamma]$ d\'esigne la classe de $\gamma$ dans $\pi'/\pi''$.

R\'eciproquement, supposons qu'il existe une famille d'homomorphismes de groupes $\varphi_i\co\pi'/\pi''\to\CC$, $1\leq i\leq n-1$, telles que $\boldsymbol{\Phi}=(\varphi_1,\ldots,\varphi_{n-1})$ v\'erifie~(\ref{Phi}) et posons
\[V(\gamma t^k)=\boldsymbol{\Phi}([\gamma])J_{n-1}^k\ ,\quad\forall\ \gamma\in\pi'\ \text{et}\ \forall\ k\in\ZZ\,.\]

Alors, pour tous $(\gamma_1t^{k_1}),(\gamma_2t^{k_2})\in\pi'\rtimes T$, on a:
\begin{align*}
V((\gamma_1t^{k_1})(\gamma_2t^{k_2}))&=V((\gamma_1+t^{k_1}\gamma_2)t^{k_1+k_2})\\
&=\boldsymbol{\Phi}([\gamma_1+t^{k_1}\gamma_2])\phi_{n-1}(\mu^{k_1+k_2})\\
&=(\boldsymbol{\Phi}([\gamma_1])+\boldsymbol{\Phi}(t^{k_1}[\gamma_2]))\phi_{n-1}(\mu^{k_1+k_2})\\
&=\boldsymbol{\Phi}([\gamma_1])\phi_{n-1}(\mu^{k_1+k_2})+\boldsymbol{\Phi}(t^{k_1}[\gamma_2])\phi_{n-1}(\mu^{k_1+k_2})\\
&=\boldsymbol{\Phi}([\gamma_1])\phi_{n-1}(\mu^{k_1+k_2})+\alpha^{k_1}\boldsymbol{\Phi}([\gamma_2])\phi_{n-1}(\mu^{k_2})\\
&=V(\gamma_1t^{k_1})\phi_{n-1}(\mu^{k_2})+\alpha^{k_1}V(\gamma_2t^{k_2})\,.
\end{align*}
Ainsi $V$ v\'erifie~(\ref{equ:hom}).
\end{proof}

\section{D\'ecomposition du module d'Alexander}
\label{sec:decomp}

Dans ce paragraphe, nous utilisons les sous-espaces vectoriels $C_n$ pour retrouver la d\'ecomposition du module d'Alexander \`a c\oe fficients complexes du n\oe ud. Notons qu'un raisonnement analogue \`a celui que nous pr\'esentons dans la suite de ce paragraphe a \'et\'e, r\'ecemment, utilis\'e par~\cite{BF08} pour montrer l'existence des repr\'esentations m\'etab\'eliennes r\'eductibles fid\`eles du groupe du n\oe ud dans $GL(n,\CC)$.

D'apr\`es~\cite{BZ85} et \cite{Gor78}, le premier groupe d'homologie $H_1(X^\infty,\CC)$ est un $\CC$-espace vectoriel de dimension finie, qui est isomorphe \`a $\pi'/\pi''\otimes\CC$. Ce dernier peut \^etre muni de la structure de $\Lambda$-module via l'action
\[t\cdot(y\otimes z)=(\mu y\mu^{-1})\otimes z\,,\quad\forall\ y\in\pi'/\pi''\ \text{et}\ \forall\ z\in\CC\,.\]
Rappelons que le module d'Alexander \`a c\oe fficients complexes d'un n\oe ud $K$ de $S^3$ se d\'ecompose sous la forme
\[
H_{1}(X^{\infty},\CC)=\displaystyle\bigoplus_{\Delta_{K}(\beta)=0}\tau_{\beta},
\text{ avec }
\tau_{\beta}=\displaystyle\bigoplus_{i=1}^{k_{\beta}}\tau_{\beta}^{i},
\text{ et }
\tau_{\beta}^{i}=\displaystyle\frac{\Lambda}{(t-\beta)^{q_{i}}},
q_{i}\in\NN^{*}
\]

Supposons maintenant que $\alpha$ est une racine du polyn\^ome d'Alexander. Puisque $\hom_\CC(\pi'/\pi''\otimes\CC,\CC_\alpha)\simeq\hom_\ZZ(\pi'/\pi'',\CC_\alpha)$, les applications $\varphi_i$ d\'ecrites dans la proposition~\ref{propphi} d\'efinissent, par extension des scalaires, des applications $\CC$-lin\'eaires $\varphi_i\co\pi'/\pi''\otimes\CC\to\CC_\alpha$ v\'erifiant~(\ref{applin}).

Le but des deux lemmes suivants est d'\'etablir un lien entre l'existence des applications $\varphi_i$ d\'ecrites dans la proposition~\ref{propphi} et les puissances $q_i$.

\begin{lemma}\label{Lemmecondnec}
Soit $n\geq 2$.

Soit $\varphi_i\co\pi'/\pi''\otimes\CC\to\CC_\alpha$, $1\leq i\leq n-1$, une famille d'applications $\CC$-lin\'eaires v\'erifiant~(\ref{applin}) et soit $\tau=\displaystyle\frac{\Lambda}{(t-\alpha)^q}$.

Si $q\leq n-2$, alors $\varphi_1|_{\tau}\equiv 0$.
\end{lemma}

\begin{proof}
Notons $\boldsymbol{\Phi}:=(\varphi_1,\ldots,\varphi_{n-1})$, alors $\boldsymbol{\Phi}$ v\'erifie~(\ref{Phi}), pour tout $y\in\pi'/\pi''\otimes\CC$,
\[\boldsymbol{\Phi}((t-\alpha)\cdot y)=\alpha\boldsymbol{\Phi}(y)(J_{n-1}^{-1}-I_{n-1})\]
et pour tout $y\in\tau$,
\[0=\boldsymbol{\Phi}((t-\alpha)^q\cdot y)=\alpha^q\boldsymbol{\Phi}(y)(J_{n-1}^{-1}-I_{n-1})^q\,.\]

Supposons que $q<n-1$, alors la $(q+1)$-i\`eme composante de $\alpha^q\boldsymbol{\Phi}(y)(J_{n-1}^{-1}-I_{n-1})^q$ est donn\'ee par $(-\alpha)^q\varphi_1(y)$. D'o\`u le r\'esultat.
\end{proof}

\begin{lemma}\label{Lemmecondsuf}
Soit $q\in\NN^{*}$ et soit $\tau=\displaystyle\frac{\Lambda}{(t-\alpha)^{q}}$, alors, pour tout $2\leq n\leq q+1$, il existe une famille d'applications $\CC$-lin\'eaires $\varphi_{i}\co\tau\to\CC_{\alpha},\ 1\leq i\leq n-1$, v\'erifiant (\ref{applin}), avec $\varphi_{1}\neq 0$.
\end{lemma}

\begin{proof}
Il suffit de montrer qu'il existe un vecteur ligne $\boldsymbol{\Phi}=(\varphi_1,\ldots,\varphi_{n-1})\co\tau\to\CC^{n-1}$ qui v\'erifie~(\ref{Phi}).

Soit $n$ un entier tel que $2\leq n\leq q+1$ et soit $\boldsymbol{\Phi}=(\varphi_1,\ldots,\varphi_{n-1})\co\tau\to\CC^{n-1}$ donn\'ee par:
\begin{displaymath}
\left\{\begin{array}{l}
\boldsymbol{\Phi}(e_0)=(1,0,\ldots,0) \\
\boldsymbol{\Phi}(e_j)=\alpha^j\boldsymbol{\Phi}(e_0)(J_{n-1}^{-1}-I_{n-1})^j\ ,\ \forall\ 1\leq j\leq q-1
\end{array}\right.
\end{displaymath}
et prolong\'ee par $\CC$-lin\'earit\'e sur $\tau$, o\`u $\{e_j=[(t-\alpha)^j];\ j=0,\ldots,q-1\}$ d\'esigne une $\CC$-base de $\tau$.

Remarquons que $\boldsymbol{\Phi}$ est bien d\'efinie puisque $q\geq n-1$ et pour tout $y\in\tau$,
\[\boldsymbol{\Phi}((t-\alpha)^q\cdot y)=\alpha^q\boldsymbol{\Phi}(y)(J_{n-1}^{-1}-I_{n-1})^q=0\,.\]
Soit $0\leq j\leq q-2$. Alors
\begin{align*}
\boldsymbol{\Phi}(t\cdot e_j)&=\boldsymbol{\Phi}(e_{j+1}+\alpha e_j)\\
&=\alpha^{j+1}\boldsymbol{\Phi}(e_0)(J_{n-1}^{-1}-I_{n-1})^jJ_{n-1}^{-1}\\
&=\alpha\boldsymbol{\Phi}(e_j)J_{n-1}^{-1}
\end{align*}
et\begin{align*}
\boldsymbol{\Phi}(t\cdot e_{q-1})&=\alpha\boldsymbol{\Phi}(e_{q-1})\\
&=\alpha^q\boldsymbol{\Phi}(e_0)(J_{n-1}^{-1}-I_{n-1})^{q-1}\\
&=\alpha^q\boldsymbol{\Phi}(e_0)(J_{n-1}^{-1}-I_{n-1})^{q-1}J_{n-1}^{-1}\\
&=\alpha\boldsymbol{\Phi}(e_{q-1})J_{n-1}^{-1}\,.
\end{align*}
Ainsi \[\boldsymbol{\Phi}(t\cdot e_j)=\alpha\boldsymbol{\Phi}(e_j)J_{n-1}^{-1}\,,\quad\forall\ 0\leq j\leq q-1\]
et\[\boldsymbol{\Phi}(t^k\cdot y)=\alpha^k\boldsymbol{\Phi}(y)J_{n-1}^{-k}\,,\quad\forall\ y\in\tau\ \text{et}\ \forall\ k\in\ZZ\,.\]
Donc $\boldsymbol{\Phi}$ v\'erifie~(\ref{Phi}).
\end{proof}

Nous avons vu, dans la proposition~\ref{propphi}, qu'il y a une correspondance entre les $\Lambda$-homomorphismes de $H_1(X^\infty,\CC)$ \`a valeurs dans $\CC_\alpha$ et le premier groupe de cohomologie $H^1(\pi,\CC_\alpha)$. Plus pr\'ecis\'ement,
\begin{equation*}
\begin{aligned}
H^{1}(\pi,\CC_{\alpha})&\simeq\hom_{\Lambda}(H_{1}(X^{\infty},\CC),\CC_{\alpha})\\
&\simeq\hom_{\Lambda}(\tau_{\alpha},\CC_{\alpha})\\
&\simeq\bigoplus_{i=1}^{k_{\alpha}}\hom_{\Lambda}(\tau_{\alpha}^{i},\CC_{\alpha})
\end{aligned}
\end{equation*}
Donc une base de $H^{1}(\pi,\CC_{\alpha})$ est donn\'ee par
$B=(\{U_{1}\},\ldots,\{U_{k_{\alpha}}\})$, o\`u  $U_{i}\co H_1(X^\infty,\CC)\to\CC_\alpha,\ 1\leq i\leq k_\alpha$ est un g\'en\'erateur de 
$\hom_{\Lambda}(\tau_{\alpha}^{i},\CC_{\alpha}).$
Par abus de notation, nous confondrons  $\{U_{i}\}$ avec le
$\Lambda$-homomorphisme correspondant. Nous pouvons alors pr\'esenter la d\'emonstration du th\'eor\`eme~\ref{thm:fond}:

\begin{proof}[Preuve du th\'eor\`eme~\ref{thm:fond}]
Soit $n\geq 2$. Pour d\'emontrer le th\'eor\`eme~\ref{thm:fond}, nous allons prouver que
\[C_n=B^1(\pi,\CC_\alpha)\oplus\vect\{U_i\ |\ q_i\geq n-1\}\,.\]
Soit $U$ un cocycle de $\pi$ dans $\CC_{\alpha}$ appartenant \`a $C_{n}$, alors, il existe un homomorphisme de groupes $\rho_{n}^{(U,v_{2},\ldots,v_{n-1})}$ de $\pi$ dans $G_n$ tel que 
$$\rho_{n}^{(U,v_{2},\ldots,v_{n-1})}(\gamma)= \left(\begin{array}{c|cccc}
\alpha^{|\gamma|}&U(\gamma)&v_{2}(\gamma)&\ldots&v_{n-1}(\gamma)\\
\hline 0&&\phi_{n-1}(\gamma)&&
\end{array}\right)\ ,\ \gamma\in\pi$$
D'apr\`es la proposition \ref{propphi}, l'existence des cocha\^ines $v_{j},\ 2\leq j\leq n-1$, est \'equivalente \`a l'existence des applications $\CC$-lin\'eaires $\varphi_{j}\co\pi'/\pi''\otimes\CC\to\CC_{\alpha}$ telles que
$$\varphi_{j}(t^{k}\cdot y)=\alpha^{k}\displaystyle\sum_{p=1}^{j}\binom{-k}{j-p}\varphi_{p}(y)\ ,\ \forall\ k\in\ZZ\ \text{et}\ \varphi_{1}=\{U\}$$
D'autre part, $\exists\ (\lambda_{i})_{1\leq i\leq k_{\alpha}}\in\CC^{k_{\alpha}}$ tel que $\{U\}=\displaystyle\sum_{i=1}^{k_{\alpha}}\lambda_{i}\{U_{i}\}$. 

Soit $1\leq i_{0}\leq k_{\alpha}$ tel que $q_{i_{0}}<n-1$, alors d'apr\`es le lemme \ref{Lemmecondnec}: $$\{U\}(y)=0,\ \forall\  y\in\tau_{\alpha}^{i_{0}}.$$\\ C'est \`a dire $\lambda_{i_{0}}\{U_{i_{0}}\}(y)=0,\ \forall\ y\in\tau_{\alpha}^{i_{0}}\ \text{car}\ \{U_{i}\}|_{\tau_{\alpha}^{j}}\equiv 0,\ \forall\ i\neq j.$\\
D'o\`u $\lambda_{i_{0}}=0.$

R\'eciproquement, soit $i\in\{1,\ldots,k_{\alpha}\}$ tel que $q_{i}\geq n-1$. D'apr\`es la proposition \ref{propphi}, pour montrer que $U_{i}\in C_{n}$, il suffit de montrer que, pour tout $2\leq j\leq n-1$, il existe une application $\CC$-lin\'eaire $\varphi_{j}\co\pi'/\pi''\otimes\CC\to\CC_{\alpha}$ telle que
$$\varphi_{j}(t^{k}\cdot y)=\alpha^{k}\displaystyle\sum_{p=1}^{j}\binom{-k}{j-p}\varphi_{p}(y),\ \forall\ k\in\ZZ\ \text{et}\ \varphi_{1}=\{U_{i}\}$$
Or l'existence de telles applications est assur\'ee par le lemme \ref{Lemmecondsuf}.

\end{proof}

\section{Exemple du n\oe ud $10_{99}$}
\label{sec:exemple}

Le premier exemple, dans le tableau de la classifiaction des n\oe uds, dont la torsion n'est ni cyclique ni semi-simple est le n\oe ud $10_{99}$. La matrice de Seifert du n\oe ud $10_{99}$ est donn\'ee par
\[V=\begin{pmatrix}
&-1&-1&0&0&0&0&-1&0&\\
&0&-1&0&0&0&0&0&0&\\
&-1&-1&-1&0&0&0&-1&0&\\
&-1&0&-1&1&0&1&0&0&\\
&-1&-1&-1&1&1&1&-1&1&\\
&0&0&0&0&0&1&0&0&\\
&0&-1&0&0&0&0&-1&0&\\
&-1&-1&-1&1&0&1&-1&1&
\end{pmatrix}\]
(voir http://www.indiana.edu/$\sim$knotinfo/)
donc une matrice de pr\'esentation du module d'Alexander est $A(t)=V^T-tV$, o\`u $V^T$ d\'esigne la matrice transpos\'ee de $V$. Le polyn\^ome d'Alexander du n\oe ud $10_{99}$ est donn\'e par $\Delta_{10_{99}}(t)=(t^2-t+1)^4$ et ses racines sont $\alpha=e^{i\pi/3}$ et $\alpha^{-1}=e^{-i\pi/3}$.

Soit $n\geq 3$. Nous cherchons un vecteur ligne $\boldsymbol{\Phi}=(\varphi_1,\ldots,\varphi_{n-1})$, o\`u les $\varphi_j\co\pi'/\pi''\otimes\CC\to\CC_\alpha$, $1\leq j\leq n-1$, sont des applications $\CC$-lin\'eaires telles que 
\[\boldsymbol{\Phi}(t^k\cdot y)=\alpha^k\boldsymbol{\Phi}(y)J_{n-1}^{-k}\,,\quad\forall\ y\in\pi'/\pi''\otimes\CC\ \text{et}\ \forall\ k\in\ZZ\,.\]

Comme $\Lambda$-module, le module d'Alexander \`a c\oe fficients complexes du n\oe ud $10_{99}$ est isomorphe \`a $\Lambda^8\big/(\Lambda^8A(t))$. Or le module $\Lambda^8$ est un $\Lambda$-module libre dont la base canonique est donn\'ee par $e_i=(0,\ldots,0,1,0,\ldots,0)$, pour $1\leq i\leq 8$. Posons $x_{ij}=\varphi_j(e_i)$, pour $1\leq i\leq 8$ et $1\leq j\leq n-1$. Alors \[\boldsymbol{\Phi}(t^k\cdot e_i)=\alpha^k\boldsymbol{\Phi}(e_i)J_{n-1}^{-k}\,,\quad\forall\ k\in\ZZ\ \text{et}\ \forall\ 1\leq i\leq 8\,,\]
o\`u $\boldsymbol{\Phi}(e_i)$ est le vecteur ligne $(x_{i1},\ldots,x_{i,n-1})$. De plus, pour que $\boldsymbol{\Phi}$ d\'efinisse une application sur $\Lambda^8\big/(\Lambda^8A(t))$ il faut et il suffit que
\begin{equation}\label{eq:module}
\boldsymbol{\Phi}(\Lambda^8 A(t))=0
\end{equation}
autrement dit,
\[\boldsymbol{\Phi}(t^ke_iA(t))=0\,,\quad\forall\ k\in\ZZ\ \text{et}\ \forall\ 1\leq i\leq 8\,.\]
Or\[\boldsymbol{\Phi}(t^ke_i(V^T-tV))=\alpha^k\boldsymbol{\Phi}(e_i(V^T-\alpha V))J_{n-1}^{-k}\]
\[+\alpha^{k+1}\boldsymbol{\Phi}(e_iV)(I_{n-1}-J_{n-1}^{-1})J_{n-1}^{-k}\,.\]
Donc~(\ref{eq:module}) est \'equivalente \`a
\begin{equation}\label{eq:cond}
(V^T-\alpha V)\boldsymbol{\varphi}+\alpha V\boldsymbol{\varphi}(I_{n-1}-J_{n-1}^{-1})=0\,,
\end{equation}
o\`u $\boldsymbol{\varphi}=(\varphi_j(e_i))_{\substack{1\leq i\leq 8\\1\leq j\leq n-1}}$.

\begin{itemize}

\item Pour $n=2$:
La condition~(\ref{eq:cond}) est \'equivalente au syst\`eme suivant:
\begin{displaymath}
\left\{\begin{array}{l}
x_{11}=\alpha x_{51}\\
x_{21}=x_{61}\\
x_{31}=-x_{51}\\
x_{41}=\alpha^2x_{51}+(\alpha-1)x_{61}\\
x_{71}=(\alpha-1)x_{61}\\
x_{81}=-\alpha x_{51}
\end{array}\right.
\end{displaymath}
Il s'en suit que $\dim H^1(\pi,\CC_\alpha)=2$, o\`u $\pi$ est le groupe du n\oe ud $10_{99}$.

\item Pour $n=3$: 
La condition~(\ref{eq:cond}) est \'equivalente \`a:
\begin{displaymath}
\left\{\begin{array}{l}
x_{11}=\alpha x_{51}\\
x_{21}=x_{61}\\
x_{31}=-x_{51}\\
x_{41}=\alpha^2x_{51}+(\alpha-1)x_{61}\\
x_{71}=(\alpha-1)x_{61}\\
x_{81}=-\alpha x_{51}\\
x_{12}=\alpha x_{51}\\
x_{22}=x_{62}+(-\alpha^2+2\alpha)x_{51}\\
x_{32}=-2x_{51}\\
x_{42}=(\alpha-1)x_{62}+\alpha^2x_{61}+(\alpha-2)x_{51}\\
x_{52}=(-2\alpha^2+\alpha)x_{61}+2x_{51}\\
x_{72}=(\alpha-1)x_{62}+(\alpha-2)x_{51}+\alpha^2x_{61}\\
x_{82}=(-2-\alpha^2)x_{61}-\alpha x_{51}
\end{array}\right.
\end{displaymath}

Ceci implique que la $(t-\alpha)$-torsion du module d'Alexander du n\oe ud $10_{99}$ se d\'ecompose sous la forme:
\[\tau_\alpha=\displaystyle\frac{\Lambda}{(t-\alpha)^{q_1}}\oplus\frac{\Lambda}{(t-\alpha)^{q_2}}\,,\ \text{avec}\ q_1,\ q_2\geq 2\,.\]

\item Pour $n=4$: 
Par un calcul direct, on montre que pour que la condition~(\ref{eq:cond}) soit satisfaite, il faut que $x_{51}=x_{61}=0$, autrement dit, il faut que $\varphi_1\equiv 0$. Par sym\'etrie, nous concluons que le module d'Alexander \`a c\oe fficients complexes du n\oe ud $10_{99}$ se d\'ecompose sous la forme:
\[\displaystyle\frac{\Lambda}{(t-\alpha)^2}\oplus\frac{\Lambda}{(t-\alpha)^2}\oplus\frac{\Lambda}{(t-\alpha^{-1})^2}\oplus\frac{\Lambda}{(t-\alpha^{-1})^2}\,.\]
\end{itemize}

\begin{center}
\textbf{Remerciements}
\end{center}

Cet article est issu de mon travail de th\`ese. Je tiens \`a remercier tr\`es vivement mes deux directeurs  Leila Ben Abdelghani et Michael Heusener pour leur soutien et tous les conseils qu'ils m'ont prodigu\'es. Leurs commentaires, leurs suggestions et leurs critiques m'ont \'et\'e tr\`es pr\'ecieux et m'ont permis d'am\'eliorer tr\`es significativement les r\'esultats que je pr\'esente dans ce papier.

\end{document}